\documentclass[11pt, oneside]{amsart}
\usepackage{mathtools}
\usepackage{wrapfig}
\usepackage{tikz}
\usepackage{comment}
\usepackage{mathrsfs}
\usepackage{tabularx, hyperref}
\usepackage{amssymb} \usepackage{amsfonts} \usepackage{amsmath}
\usepackage{amsthm} \usepackage{epsfig, subfig}
\usepackage{ amscd, amsxtra, latexsym}
\usepackage{epsfig,  graphicx, psfrag}
\usepackage[all]{xy}
\usepackage{caption}
\usepackage{enumerate}
\usepackage{color}

\DeclareMathOperator\supp{Supp}

\DeclareMathOperator\lk{lk}

\newtheorem{lemma}{Lemma}[section]
\newtheorem{thm}[lemma]{Theorem}
\newtheorem{prop}[lemma]{Proposition}
\newtheorem{cor}[lemma]{Corollary}
\newtheorem{conj}[lemma]{Conjecture}

\newtheorem*{prop*}{Proposition}
\newtheorem{prop_intro}{Proposition}

\newtheorem{thm_intro}[prop_intro]{Theorem}

\newtheorem{cor_intro}[prop_intro]{Corollary}
\theoremstyle{definition}

\newtheorem{rem}[lemma]{Remark}

\theoremstyle{definition}

\newtheoremstyle{citing}
  {3pt}
  {3pt}
  {\itshape}
  {}
  {\bfseries}
  {}
  {.5em}
  {\thmnote{#3}}
\theoremstyle{citing}

\definecolor{darkgreen}{cmyk}{1,0,1,.2}

\DeclareMathOperator{\minvol}{MinVol}
\DeclareMathOperator{\vol}{Vol}
\DeclareMathOperator{\inte}{int}

\newcommand{\lf}{\text{lf}}

\newcommand{\R} {\ensuremath {\mathbb{R}}}

\newcommand{\matZ} {\ensuremath {\mathbb{Z}}}

\renewcommand{\phi}{\varphi}

\begin{document}

\title[]{The simplicial volume of contractible 3-manifolds}

\author[Giuseppe Bargagnati]{Giuseppe Bargagnati}
\address{Dipartimento di Matematica, Universit\`a di Pisa, Italy}
\email{g.bargagnati@studenti.unipi.it}


\author[]{Roberto Frigerio}
\address{Dipartimento di Matematica, Universit\`a di Pisa, Italy}
\email{roberto.frigerio@unipi.it}

\date{\today}

\begin{abstract}
We show that the simplicial volume of a contractible $3$-manifold not homeomorphic to $\mathbb{R}^3$ is infinite. As a consequence, the Euclidean space may be characterized
as the unique contractible $3$-manifold with vanishing minimal volume, or as the unique contractible $3$-manifold supporting
a complete finite-volume Riemannian metric with Ricci curvature uniformly bounded from below. On the contrary, we show that in every dimension $n\geq 4$ there 
exists a contractible $n$-manifold  with vanishing simplicial volume not homeomorphic to  $\mathbb{R}^n$.
We also  compute the spectrum of the simplicial volume of irreducible open $3$-manifolds. 
\end{abstract}

\maketitle

\section{Introduction}
The simplicial volume is a homotopy invariant of manifolds introduced by Gromov in his pioneering paper~\cite{Gromov} (see Section~\ref{simplicial:sec} for the precise definition).
As highlighted by Gromov himself,
its behaviour is deeply related
to the geometric structures that a manifold can carry. 
For compact manifolds, the study of the simplicial volume has often exploited the dual theory of bounded cohomology,
which however seems to be a bit less effective 
 for the computation of  the simplicial volume of \emph{open} manifolds
 (but see~\cite{Gromov, Loeh, FriMoM} for several applications of bounded cohomology to open manifolds).
This is probably one of the reasons why the investigation of the simplicial volume 
is apparently even more challenging
for
open manifolds than for closed ones. 

Vanishing and finiteness results for the simplicial volume of open manifolds are obtained in~\cite{Gromov,Loeh,FriMoM}, and the exact value of the simplicial volume of some open locally symmetric spaces
is computed in~\cite{Loh-Sauer, Loh-Sauer2,KimKim,MichelleKim}. Building on~\cite{HL}, recently Heuer and L\"oh proved that, for any fixed $n\geq 4$ and any $\alpha\in\mathbb{R}$,
there exists an $n$-dimensional open manifold whose simplicial volume is equal to $\alpha$~\cite{HL2}. In dimension 3, as stated in~\cite{HL2}, 
 the computation of all the possible simplicial volumes of open manifolds ``seems to be fairly tricky'',
  since the structure of non-compact 3-manifolds can get quite complicated.

Henceforth, by ``manifold'' we will understand a topological (second countable) manifold (recall however that, in dimension 3, every manifold admits a unique 
smooth structure up to diffeomorphism).  A manifold is closed if it is connected, compact and without boundary, and open if it is  connected, non-compact and without boundary.
Every manifold will be assumed to be connected and orientable (the simplicial volume may be defined also for non-orientable manifolds by passing 
to the $2$-sheeted orientable cover, and for disconnected manifolds by summing the simplicial volumes of the connected components). 
If $M$ is a manifold without boundary, we denote by $\|M\|$ the simplicial volume of $M$.

\subsection*{Contractible $3$-manifolds}
The main result of this paper completely describes the simplicial volume of \emph{conctractible} $3$-manifolds, showing that the Euclidean space may be characterized as the unique 
contractible $3$-manifold whose simplicial volume vanishes:

\begin{thm_intro}\label{main:thm}
Let $M$ be a contractible 3-manifold. Then $\|M\|=0$ if $M$ is homeomorphic to $\mathbb{R}^3$, and $\|M\|=+\infty$ otherwise.
\end{thm_intro}

For example, the famous Whitehead manifold $W$~\cite{White} has infinite simplicial volume. Gabai showed in~\cite{GabaiW} that $W$ admits a decomposition
 into open subsets $A,B$ such that each of $A$, $B$ and $A\cap B$ is homeomorphic to a copy of the Euclidean space. 
 This property was then called ``double 3-space property'' in~\cite{GRW}, where it is shown that there exists (up to homeomorphism) uncountably many 
 contractible $3$-manifolds with the the double 3-space property, as well as uncountably many without this property.
 Our result implies the (surprising at first sight) fact that,
even if $\|\mathbb{R}^3\|=0$, the simplicial volume of a manifold $M$ with the double 3-space property does not  vanish, unless $M\cong \mathbb{R}^3$. We refer the reader e.g.~to~\cite{Mcmillan1,Mcmillan2}
for a survey on the topology of contractible $3$-manifolds.

In order to obtain Theorem~\ref{main:thm} we prove two results which may be of independent interest: in Proposition~\ref{amenable:prop}
we find sufficient conditions under which the isomorphism provided by the Excision Theorem for singular homology is isometric with respect to the $\ell^1$-norm,
while in Theorem~\ref{lkgraphlink} we show that the components of a non-split $2$-component link  $L\subseteq S^3$ containing a trivial knot and
such that $\|S^3\setminus L\|=0$ must be algebraically linked.

As already mentioned, Gromov himself pointed out that the simplicial volume may be exploited to obtain lower bounds on differential-geometric invariants of Riemannian manifolds.
Let us list some consequences of Theorem~\ref{main:thm} in this spirit.
If $M$ is a smooth manifold, the \emph{minimal volume} $\minvol (M)$ of $M$ is defined as 
the greatest lower bound of the  volumes $\vol (M,g)$ of $M$ with respect to complete Riemannian metrics $g$ whose sectional curvature $K(g)$ satisfies $-1\leq K(g)\leq 1$.
It is known that, in higher dimension, the minimal volume may distinguish distinct smooth structures on the same topological manifold~\cite{bess}: however, in dimension 3 every topological manifold
supports a smooth structure, and homeomorphic smooth manifolds are in fact diffeomorphic, hence the minimal volume is an invariant of topological $3$-manifolds.  
A fundamental result by Gromov ensures that the simplicial volume provides a lower bound on the minimal volume (up to a universal constant only depending on the dimension of the manifold): indeed,
if $M$ is an open $n$-dimensional smooth manifold, then
$$
\|M\|\leq (n-1)! n^n \minvol(M)
$$
(see~\cite[page 12]{Gromov}). 
Moreover, 
Gromov himself proved that the minimal volume of $\mathbb{R}^3$ vanishes~\cite[Appendix B]{Gromov} (see also~\cite[Examples 1.4,1.5,1.6]{CheeGro}, where Cheeger and Gromov proved
that $\minvol(\mathbb{R}^n)=0$ for every $n\geq 3$). Thus Theorem~\ref{main:thm} implies the following corollary, which provides a characterization of $\mathbb{R}^3$ among contractible $3$-manifolds in terms of the minimal volume.

\begin{cor_intro}\label{minimalvolume:cor}
 Let $M$ be a contractible 3-manifold. Then $\minvol (M)=0$ if $M$ is homeomorphic to $\mathbb{R}^3$, and $\minvol (M)=+\infty$ otherwise.
\end{cor_intro}

The Euclidean space may be  characterized also as the unique contractible $3$-manifold admitting a complete finite-volume Riemannian metric with Ricci curvature uniformly bounded from below. 
In fact, Gromov's \emph{Main Inequality}~\cite[page 12]{Gromov} asserts that, for every complete Riemannian metric $g$ on the $n$-dimensional manifold $M$,  if 
$$\text{Ricci}(g)\geq -\frac{1}{n-1} g\ ,$$
then
$$
\|M\|\leq n! \vol(M)\ .
$$
Thus Theorem~\ref{main:thm} implies the following:

\begin{cor_intro}\label{ricci:cor}
Let $M$ be a  contractible 3-manifold. If $M$ supports a complete finite-volume Riemannian metric  with Ricci curvature uniformly bounded from below, then $M$ is 
homeomorphic to $\mathbb{R}^3$.
\end{cor_intro}
\begin{proof}
If $M$ supports a complete Riemannian metric $g$ with $\text{Ricci}(g)\geq -k g$, $k>0$, up to rescaling we may assume that $k=1/(n-1)$. Then
Gromov's Main Inequality applies to show that, if $\|M\|=+\infty$, then also the volume of $g$ should be infinite. The conclusion now follows from Theorem~\ref{main:thm}.
\end{proof}

For other characterizations of the Euclidean space among contractible $3$-manifolds in terms of curvature we address the reader e.g.~to~\cite{Wein2},
where it is proved that a contractible $3$-manifold admitting a complete metric with uniformly positive scalar curvature must be diffeomorphic to $\mathbb{R}^3$,
or to~\cite{Wang}, where the same result is proved under the weaker hypothesis that  scalar curvature  be positive
with slower-than-quadratic decay, or to~\cite{Zhu}, where the author shows that a complete, noncompact three-manifold with nonnegative Ricci curvature is diffeomorphic to
the  Euclidean space if the Ricci curvature is strictly positive at some point. Very recently, these results were superseded by~\cite{JWang}, where Wang proved that any complete contractible 3-manifold with non-negative scalar curvature is homeomorphic to $\mathbb{R}^3$.

\subsection*{The higher dimensional case}
Theorem~\ref{main:thm} does not hold in dimension $\geq 4$. In fact, in Section~\ref{higher:sec} we prove the following:

\begin{thm_intro}\label{higher:dim:thm}
Let $n\geq 4$. Then, there exists a contractible $n$-manifold which is not homeomorphic to $\mathbb{R}^n$ and is such that $\|M\|=0$.
Moreover, $M$ can be chosen to be smooth.
\end{thm_intro}

\subsection*{The simplicial volume of irreducible $3$-manifolds}
Following~\cite{HL2}, we denote by $SV(3)$ the set of simplicial volumes of \emph{closed} $3$-manifolds, and by $SV^{\text{lf}}(3)$ the set of simplicial
volumes of all $3$-manifolds without boundary. Recall that an open manifold is \emph{tame} if it is homeomorphic to the internal part of a compact manifold with boundary,
and denote by $SV^{\text{lf}}_{\text{tame}}(3)$ the set of simplicial volumes of tame open $3$-manifolds. It is known that $SV(3)$ is countable
and exhibits a gap at $0$, i.e.~there exists $\varepsilon>0$ such that $[0,\varepsilon)\cap SV(3)=\{0\}$ (see e.g~\cite[Corollary 7.8]{frigerio:book}; in fact, one may choose
$\varepsilon=\vol(M_0)/v_3\approx 0.928\dots$, where $M_0$ is the Weeks manifold, i.e.~the 
complete orientable hyperbolic 3-manifold of smallest volume~\cite{GMM}, and $v_3$ is the volume of the regular ideal tetrahedron in hyperbolic
$3$-space). Moreover, it is shown in~\cite{HL2} that $SV^{\text{lf}}_{\text{tame}}(3)=SV(3)\cup\{\infty\}$. However,
the set of simplicial volumes of open $3$-manifolds is not understood completely yet (see~\cite[Question 1.3]{HL2}).
Here we propose the following:

\begin{conj}\label{nontame:conj}
$SV^{\text{lf}} (3)=SV(3)\cup \{\infty\}$.
\end{conj}

Recall that a $3$-manifold $M$ is \emph{irreducible} if every embedded $2$-sphere $S\subseteq M$ bounds a closed $3$-ball $B\subseteq M$.
We denote 
by $SV^{\text{lf}}_{\text{irr}}(3)$ the set of
simplicial volumes of (open or closed) irreducible $3$-manifolds. 
As a partial step towards the proof of Conjecture~\ref{nontame:conj} we prove here the following:

\begin{thm_intro}\label{spectrum:thm}
$SV^{\text{lf}}_{\text{irr}}(3)\subseteq SV(3)\cup \{\infty\}$.
\end{thm_intro}

Theorem~\ref{spectrum:thm} may be exploited to show that $\|M\|\in SV(3)\cup\{\infty\}$ for every open $3$-manifold which admits a decomposition along spheres 
into a (possibly infinite) union of (possibily non-compact) irreducible $3$-manifolds. However, there are examples of open $3$-manifolds which do
not admit such a decomposition~\cite{Scottloc,Maillot}.

\subsection*{Plan of the paper} In Section~\ref{simplicial:sec} we collect some basic facts about simplicial volume and bounded cohomology, and we prove 
Proposition~\ref{amenable:prop} on the behaviour of the $\ell^1$-norm under excision. Section~\ref{contractible:sec} is devoted to the study of the geometry
of open $3$-manifolds. A particular attention is payed to the r\^ole played by exhaustions, and to the case of contractible manifolds. In Section~\ref{link:sec}
we prove Theorem~\ref{lkgraphlink} about $2$-component links with vanishing simplicial volume, and in Section~\ref{proof:sec} we put together the results of the previous
sections to prove Theorem~\ref{main:thm}. The higher dimensional case (Theorem~\ref{higher:dim:thm})  is addressed in Section~\ref{higher:sec}, while in the last section
we prove Theorem~\ref{spectrum:thm} on the spectrum of the simplicial volume of open irreducible $3$-manifolds.

\subsection*{Acknowledgements} The authors thank Bruno Martelli for useful conversations.

\section{Simplicial volume and bounded cohomology}\label{simplicial:sec}
Let $X$ be a topological space, and let $R=\matZ,\R$. 
For every $n\in\mathbb{N}$, we denote by
$S_n (X)$ the set of singular $n$--simplices
in $X$. 

A subset $A \subset S_n(X)$ is \emph{locally finite} if any compact set $K \subseteq X$ intersects the image of only finitely many singular simplices of $A$. 
Following \cite[Chapter 5.1]{Lothesis}, a (possibly infinite) singular $n$-chain on $X$ (with coefficients in $R$) is a formal sum $\sum_{\sigma\in S_n(X)} a_\sigma \sigma$, where 
every $a_\sigma$ is an element of $R$. We say that such a chain
is \emph{locally finite} if the set $\{\sigma\in S_n(X)\, |\, a_\sigma\neq 0\}$ is locally finite, and we denote by $C_n^{\lf}(X;R)$ the $R$-module of locally finite chains on $X$.

The usual boundary operator on finite chains may be extended to the complex of locally finite chains, so it makes sense
to define the locally finite homology $H^{\lf}_*(X;R)$ of $X$ as the homology of the complex $C_*^{\lf}(X;R)$. Of course, if $X$ is compact,
then locally finite chains are finite, and we recover the usual complex $C_*(X;R)$ of singular chains with values in $X$, and the usual singular homology
module $H_*(X;R)$.


Henceforth, unless otherwise stated, 
when we omit the coefficients from our notation we understand that $R=\mathbb{R}$.
All the (bounded) (co)homology modules will be understood with real coefficients. For example, we will denote the vector spaces $C_*^\lf(X;\R)$ and $ H^{\lf}_*(X;\R)$ simply by $C_*^\lf(X)$ and $H_*^\lf(X)$. 

The space $C_n^\lf(X)$ may be endowed with the $\ell^1$-norm defined by
$$
\left\|\sum_{\sigma \in S_n(X)} a_\sigma \sigma\right\|_1=\sum_{\sigma \in S_n(X)} |a_\sigma|\ \in\ [0,+\infty]\ .
$$
This norm induces
an $\ell^1$-seminorm (with values in $[0,+\infty]$) on $H_*^\lf(X)$, which will still be denoted by $\|\cdot\|_1$.

It is a standard result of algebraic topology (see for instance \cite[Theorem~5.4]{Lothesis}) that, if $M$ is an $n$-dimensional connected and oriented manifold, then $H^{\lf}_n(M;\mathbb{Z})\cong \mathbb{Z}$ is generated by a preferred
element $[M]_\mathbb{Z}\in H^{\lf}_n(M;\mathbb{Z})$, called the \emph{fundamental class} of $X$. Under the obvious change of coefficients homomorphism
$H^{\lf}_*(X;\mathbb{Z})\to H^{\lf}_*(M)$, the element $[M]_\mathbb{Z}$ is taken  to the \emph{real} fundamental class $[M]\in H^{\lf}_n(M)$ of $M$.

If  $M$ is an oriented $n$-manifold, the \emph{simplicial volume} of $M$ is $$\|M \| \coloneqq \| [M] \|_{1}\ .$$ 


We will be also interested in the case when $M$ is compact with boundary. In this case, the $\ell^1$-norm on $C_n(M)$ induces an $\ell^1$-norm on the relative
singular chain module $C_n(M,\partial M)$, which in turn induces a seminorm on $H_n(M,\partial M)$. If $M$ is connected, oriented and $n$-dimensional,
then $H_n(M,\partial M;\mathbb{Z})\cong \mathbb{Z}$, and the positive generator $[M,\partial M]_\mathbb{Z}\in H_n(M,\partial M;\mathbb{Z})\cong \mathbb{Z}$
is sent by the change of coefficient map to the \emph{relative fundamental class} $[M,\partial M]\in H_n(M,\partial M;\mathbb{R})= H_n(M,\partial M)$. Just as in the case without boundary, the simplicial volume of $M$ is then $$\|M,\partial M\|=\|[M,\partial M]\|_1\ .$$

In the following lemmas we collect some elementary properties of the simplicial volume which will prove useful later. 

\begin{lemma}\label{general1:lemma}
Let $M$ be a compact $3$-manifold with boundary.
If $\|M,\partial M\|=0$, then for every $\varepsilon>0$ there exists a fundamental cycle $z\in C_3(M,\partial M)$ such that
$\|z\|_1\leq \varepsilon$ and $\|\partial z\|_1\leq \varepsilon$.\end{lemma}
\begin{proof}
Let $\varepsilon>0$ be given. 
By definition, if $\|M,\partial M\|=0$, then  there exists a fundamental cycle $z\in C_3(M,\partial M)$ such that $\|z\|_1\leq \varepsilon/4\leq \varepsilon$. 
Since a $3$-simplex has $4$ facets, the boundary operator $\partial\colon C_{3}(M,\partial M)\to C_{2}(\partial M)$ has norm at most $4$, hence $\|\partial z\|_1\leq \varepsilon$.
\end{proof}

\begin{lemma}\label{general2:lemma}
There exists a constant $k>0$ such that the following holds. 
If $M$ is a compact $3$-manifold whose boundary is given by a finite union of spheres and/or tori, and $\|M,\partial M\|\neq 0$, then $\|M,\partial M\|\geq k$.
\end{lemma}
\begin{proof}
 Since $M$ is bounded by spheres and tori, we have 
$\|M,\partial M\|=\|\inte(M)\|$, hence $\|M,\partial M\|\in SV^{\text{lf}}_{\text{tame}}(3)\subseteq SV(3)\cup \{\infty\}$. Now the conclusion
follows e.g.~from~\cite[Example 2.5]{HL}.

\end{proof}

Let now $M$ be an $n$-manifold, and let $N\subseteq M$ be a  compact connected codimension-0 submanifold. 
An easy application of the Excision Theorem for singular cohomology (together with the fact that $\partial N$ admits a collar in $N$, hence it is a strong neighbourhood deformation retract of $N$)
shows that, for every $k\in\mathbb{N}$, the map
$$
\psi_k\colon H_k (N,\partial N)\to H_k(M,M\setminus \inte(N))
$$
induced by the inclusion $(N,\partial N)\to (M,M\setminus \inte(N))$ is an isomorphism. In order to prove Theorems~\ref{main:thm} and~\ref{spectrum:thm},
we will need to show that, if $c$ is a fundamental cycle for an open manifold $M$, 
and $N$ is a suitably chosen codimension-0 compact submanifold of $M$,
then the finite chain $c|_N$ obtained by considering only the simplices of $c$ intersecting $N$
has an $\ell^1$-norm not smaller than $\|N,\partial N\|$. This is directly related to the fact that the isomorphism $\psi_k$ is isometric with respect to the $\ell^1$-norm. 
However, as shown in Remark~\ref{counter:ex} below,
the isomorphism $\psi_k$ is very far from being isometric in general. As is customary when working with the simplicial volume, when performing gluings along boundary components,
it is convenient (and somewhat necessary) to assume that such components are $\pi_1$-injective and have an amenable fundamental group (see e.g.~\cite{BBFIPP}):

\begin{prop}\label{amenable:prop}
Let $M$ be an $n$-manifold, and let $N\subseteq M$ be a  compact connected codimension-0 submanifold. 
Suppose that every component  of 
$\partial N$ is $\pi_1$-injective in $M$, and has an amenable fundamental group.
Then, for every $k\geq 2$ the map
$$
\psi_k\colon H_k (N,\partial N)\to H_k(M,M\setminus \inte(N))
$$
induced by the inclusion $(N,\partial N)\to (M,M\setminus \inte(N))$ is an isometric isomorphism.
\end{prop}
\begin{proof}
As already mentioned, the fact that $\psi_k$ is an isomorphism is a consequence of the excision property for singular homology. Moreover, being induced
by a norm non-increasing map on singular chains, the map $\psi_k$ is norm non-increasing. In order to show that $\psi_k$ is also norm non-decreasing we need to exploit
the duality between the $\ell^1$-norm on singular cohomology and the $\ell^\infty$-norm on {bounded cohomology}, which we now briefly describe.

If $(X,Y)$ is any pair of spaces (i.e.~$X$ is a topological space and $Y\subseteq X$), we denote by $C^*(X,Y)$ the complex of singular relative cochains with real coefficients. 
For every $i\in\mathbb{N}$, we may
regard $S_i (X)$ as a subset of $C_i (X)$, so that
for any cochain $\varphi\in C^i (X,Y)$ it makes sense to set
$$
\|\varphi\|_\infty  = \sup \left\{|\varphi (s)|\ |\ s\in S_i (X)\right\}\in [0,\infty].
$$
We denote by $C^*_b (X,Y)$ the submodule of bounded cochains, i.e.~we set
$$C^*_b (X,Y)= \left\{\varphi\in C^* (X,Y)\ | \ \|\varphi\|<\infty\right\}\ .$$ Since
the differential takes bounded cochains to bounded cochains, $C^*_b (X,Y)$
is a subcomplex of $C^*(X,Y)$. The cohomology $H^*_b(X,Y)$ of the complex $C^*_b (X,Y)$ is the \emph{bounded cohomology} of the pair $(X,Y)$,
and is endowed with the 
$\ell^\infty$-seminorm induced by the $\ell^\infty$-norm on $C^*_b(X,Y)$.

The $\ell^\infty$-norm on singular cochains coincides with the dual norm of the $\ell^1$-norm on chains, and 
the duality pairing between $C^n_b(X,Y)$ and $C_n(X,Y)$ induces the \emph{Kronecker product}
$$
\langle \cdot,\cdot \rangle\colon H_b^n(X,Y)\times H_n(X,Y)\to \R\ .
$$
An easy argument based on Hahn-Banach Theorem  (see e.g.~\cite[Lemma 6.1]{frigerio:book}) implies 
that, for every 
 $n\in\mathbb{N}$ and  $\alpha\in H_n(X,Y)$, 
 \begin{equation}\label{duality:eq}
\|\alpha\|_1=\max \{ \langle\beta,\alpha\rangle\, |\,  \beta\in H_b^n(X,Y),\,  \|\beta\|_\infty \leq 1\}\ .
\end{equation}

Let us now take a class $\alpha\in H_k(N,\partial N)$, and let $z\in C_k(N)$ be a representative of $\alpha$
(hence $\partial z\in C_{k-1}(\partial N)$). 
By~\eqref{duality:eq},
there exists a bounded cohomology class $\varphi\in H^k_b(N,\partial N)$ such that $\|\varphi\|_\infty\leq 1$, and
$$
\langle \varphi,\alpha\rangle =\|\alpha\|_1\ .
$$
Let $\varepsilon>0$ be given, and
let $f\in C^k_b(N,\partial N)$ be a cocycle representing $\varphi$ and such that $\|f\|_\infty\leq\|\varphi\|_\infty+\varepsilon\leq 1+\varepsilon$.

Our hypotheses on the boundary components of $N$ allow us to 
assume that $f$ is \emph{special} (see~\cite[Definition 5.15 and Corollary 5.18]{frigerio:book}), and to 
extend $f$ to a relative cocycle $g\in C^k(M,M\setminus \inte{N})$ such that
$\|g\|_\infty=\|f\|_\infty$: indeed, if $C_1,\dots,C_h$ are the components of $M\setminus \inte{N}$,
then we
set $g_j\in C^k_b(C_j)$, $g_j=0$  for every $j=1,\dots,h$, and we then extend the collection of cocycles $\{f,g_1,\dots,g_h\}$ to a global
cocycle $g$ with $\|g\|_\infty=\|f\|_\infty$ as described in the proof of~\cite[Theorem 9.3]{frigerio:book}. By construction, $g$ vanishes on $M\setminus \inte{N}$, hence it is indeed 
an element of $C^k_b(M,M\setminus \inte{N})$.

Let now $z'\in C_k(M,M\setminus \inte(N))$ be any representative of $\psi_k(\alpha)$, so that $z'=z+\partial c+ d$, where $c\in C_{k+1}(M)$ and
$d\in C_{k}(M\setminus \inte(N))$ 
(here we are considering $C_k(N)$ as a subset of $C_k(M)$).
Since $g$ is a relative cocycle in $C^n_b(M,M\setminus \inte{N})$, we have 
$$\langle g, z'\rangle=\langle g,z\rangle=\langle f,z\rangle=\langle \varphi,\alpha\rangle =\|\alpha\|_1\ .$$ 
Moreover, 
$$
\langle g, z'\rangle\leq \|g\|_\infty \|z'\|_1=\|f\|_\infty\|z'\|_1 \leq \left( 1+\varepsilon\right) \|z'\|_1\ .
$$
We thus get $\|z'\|_1\geq \|\alpha\|_1/(1+\varepsilon)$, whence $\|z'\|_1\geq \|\alpha\|_1$ due to the arbitrariness of $\varepsilon$.
Since this inequality holds for every representative $z'$ of $\psi_k(\alpha)$, we may deduce that $\|\psi_k(\alpha)\|_1\geq \|\alpha\|_1$, and this concludes the proof of the proposition.
\end{proof}

\begin{rem}\label{counter:ex}
The following examples show that the assumption that every component of $\partial N$ is $\pi_1$-injective and has an amenable fundamental group
is indeed necessary for Proposition~\ref{amenable:prop} to hold.

(1): 
Let $M=S^3$, let $K\subseteq M$ be a hyperbolic knot, and let $N$ be the complement in $M$ of an open regular neighbourhood of $K$. 
Then $\partial N$ consists of a torus, hence $\pi_1(\partial N)$ is amenable (but $\partial N$ is not $\pi_1$-injective in $M$). 
Let $\pi\colon \colon H_3(M)\to H_3(M, M\setminus \inte(N))$  be the map induced
by the quotient map $C_3(M)\to C_3(M, M\setminus \inte(N))$. It is immediate to check that, if $\psi_3$ is as in Proposition~\ref{amenable:prop}, then
$\psi_3([N,\partial N])=\pi([M])$. However, $\|[N,\partial N]\|_1=\|N,\partial N\|_1>0$ since $K$ is hyperbolic, while
$\|\pi([M])\|_1\leq \|[M]\|_1=\|S^3\|=0$, since $\pi$ is norm non-increasing. This shows that $\psi_3$ is not an isometry.

(2):
Let $M$ be a closed orientable $3$-manifold containing a sequence of separating $\pi_1$-injective connected  surfaces $\{S_n\}_{n\in\mathbb{N}}$ of arbitrarily high genus
(such a manifold exists e.g.~by~\cite{Koba}). Let $N_n$ be (the closure of) one component of $M\setminus S_n$. By~\cite[Theorem 1]{BFP} we have
 $$\|N_n,\partial N_n\|\geq  (3/4)\|\partial N_n\|=(3/2)|\chi( S_n)|\ ,$$ hence $$\lim_{n\to +\infty} \|N_n,\partial N_n\|=+\infty\ .$$ 
Now, if $\pi_n\colon \colon H_3(M)\to H_3(M, M\setminus \inte(N_n))$  is  induced
by the quotient  $C_3(M)\to C_3(M, M\setminus \inte(N_n))$,
and $\psi_{n,3}\colon H_3(N_n,\partial N_n)\to H_3(M,M\setminus \inte{N_n})$ is as in Proposition~\ref{amenable:prop},
 then as above
$\psi_{n,3}([N_n,\partial N_n])=\pi_n([M])$. If $\psi_{n,3}$ were isometric for every $n\in\mathbb{N}$, then we would have
$$
\|N_n,\partial N_n\|=\|\psi_{n,3}([N_n,\partial N_n])\|_1=\|\pi_n([M])\|_1\leq \|[M]\|_1=\|M\|\ ,
$$
which would contradict the fact that $\lim_{n\to +\infty} \|N_n,\partial N_n\|=+\infty$.
\end{rem}

\section{Open irreducible $3$-manifolds}\label{contractible:sec}
Recall from the introduction that a $3$-manifold $M$ is irreducible if  every embedded $2$-sphere $S\subseteq M$ bounds a closed $3$-ball $B\subseteq M$.
Moreover, a surface $S\subseteq M$ is \emph{compressible} if there exists a simple closed loop $\gamma\subseteq S$ which does not bound any $2$-dimensional
disc on $S$, while it bounds a $2$-dimensional
disc $D\subseteq M$ such that $D\cap S=\gamma$. We will only consider compact connected orientable surfaces without boundary in orientable $3$-manifolds; in this context, 
the Dehn's Lemma (proved by Papakyriakopoulos~\cite{papa})
ensures that $S$ is incompressible if and only if it is $\pi_1$-injective, i.e.~if and only if the inclusion $S\hookrightarrow M$ induces
an injective map $\pi_1(S)\to \pi_1(M)$ on fundamental groups.

Let $M$ be a connected open $3$-manifold. We say that a subset of $M$ is \emph{bounded} if it is relatively compact, and \emph{unbounded} otherwise.
An \emph{exhaustion} of $M$ is  a sequence $\{M_n\}_{n\in\mathbb{N}}$ of connected  subsets of $M$ such that
$M=\bigcup_{n\in\mathbb{N}} M_n$ and,
for every $n\in\mathbb{N}$,  the set $M_n$ is a 
compact $3$-manifold with boundary,  
 every component of $M \setminus \inte{M_n}$ is unbounded, and
$M_n\subseteq \inte(M_{n+1})$.

\begin{lemma}\label{regular}
Every open $3$-manifold $M$ admits an exhaustion $\{M_n\}_{n\in\mathbb{N}}$  such that, for every $n\in\mathbb{N}$, every
component of $M\setminus \inte(M_n)$ has exactly one boundary component.
\end{lemma}
\begin{proof}
It is well-known that
there exists a proper smooth function $f\colon M\to \mathbb{R}$. By Morse-Sard Theorem, we may choose a sequence $\{R_n\}_{n\in\mathbb{N}}\subseteq \mathbb{R}$ 
of regular values for $f$ such that
$\lim_{n\to \infty} R_n=+\infty$.
Then, for every $n\in\mathbb{N}$ the set $M_n=f^{-1}([-R,R])$ is  a compact smooth manifold with boundary. 
Observe that, being compact, the manifold $M_n$ has a finite number of connected components.
Up to 
adding to $M_n$ the disjoint regular neighbourhoods of a finite number of disjoint arcs (transverse to $\partial M_n$) connecting the components of $M_n$, we may suppose that $M_n$ is connected,
and up to
replacing $M_n$ with the union of $M_n$ with all the bounded components
of $M\setminus M_n$, we may also assume that $M\setminus M_n$ has no bounded components.

Let now $V$ be a component of $M_n\setminus \inte (M_n)$. Suppose that
$\partial V$ is disconnected, and denote by $B_1,\dots,B_k$ the connected components of $\partial V$ (the components of $\partial V$ are in finite
number because $\partial V\subseteq \partial M_n$ and $M_n$ is compact). For every $i=1,\dots,k-1$, choose an arc $\gamma_i\subseteq V$
joining a point in $B_i$ with a point in $B_{i+1}$, in such a way that the $\gamma_i$ are pairwise disjoint. Finally, take small disjoint regular neighbourhoods $N_1,\dots,N_{k-1}$
of $\gamma_1,\dots,\gamma_{k-1}$, and replace  $M_n$ with $M_n\cup N_1\cup\dots\cup N_{k-1}$. After repeating this procedure for every component $M_n\setminus \inte (M_n)$
with disconnected boundary, we end up with a compact connected submanifold with boundary (which we still denote by $M_n$) with the following properties: 
$f^{-1}([-R_n,R_n])\subseteq M_n$, and every component of $M\setminus \inte(M_n)$ is unbounded and has exactly one boundary component. Up to passing to a subsequence 
$\{M_{n_i}\}_{i\in\mathbb{N}}$, we may also obtain $M_{n_i}\subseteq \inte (M_{n_{i+1}})$ for every $i\in\mathbb{N}$, and this concludes the proof.
\end{proof}

If $\sigma\colon \Delta_n\to M$ be a singular simplex,  we denote by $\supp(\sigma)\subseteq M$ the image of $\sigma$.
If $c=\sum_{i\in I} a_i\sigma_i$ is a locally finite singular chain in $M$ and $X$ is a subset of $M$, then we set
$\supp(c)=\bigcup_{i\in I} \supp(\sigma_i)$, and
 $c|_X=\sum_{i\in I_X} a_i\sigma_i$, where $I_X=\{i\in I\, |\, \supp(\sigma_i)\cap X\neq \emptyset\}$. 
It follows from the very definitions that, if $c$ is locally finite and $X$ is compact, then $c|_X$ is a finite chain, hence $\supp(c|_X)$ is also compact
(but observe that $\supp(c|_X)\nsubseteq X$ in general).

\begin{prop}\label{toric}
Let $M$ be an open $3$-manifold with $\|M\|<+\infty$. Then there exists an exhaustion  $\{M_n\}_{n\in\mathbb{N}}$ of $M$ such that, for every $n\in\mathbb{N}$,
the boundary
$\partial M_n$ is a finite union of spheres and/or tori.
\end{prop}
\begin{proof}
By Lemma~\ref{regular}, we may take an exhaustion
 $\{Z_n\}_{n\in\mathbb{N}}$ of $M$
 such that, for every $n\in\mathbb{N}$, every
component of $M\setminus \inte{Z_n}$ has exactly one boundary component.
 

Let us fix $n\in\mathbb{N}$. We first construct a compact submanifold $M'_n\subseteq M$ such
that $Z_n\subseteq M'_n$ and $M'_n$ is bounded by spheres and/or tori.

Let $c\in C_3(M)$ be a locally finite fundamental cycle for $M$ with finite $\ell^1$-norm. For every $i\in\mathbb{N}$, also set $c_i=c|_{Z_i}$.  
We have $\lim_{m\to +\infty} \|c-c_m\|=0$, hence there exists $m>n$ such that $\|c-c_{m}\|<1/2$. Recall now that the boundary map
$\partial \colon C_3(M)\to C_2(M)$ has operator norm equal to $4$, and that $\partial c_{m}=\partial (c-c_{m})$ since $c$ is a cycle. We thus
have
$\|\partial c_m\|_1=\|\partial (c-c_m)\|_1\leq 4\|c-c_m\|_1<2$. Also observe that, since $\supp(c_m)$ is compact, there exists $m'>m$ such that
$\supp(c_m)\subseteq Z_{m'}$.

By construction, $\supp(\partial c_m)=\supp (\partial (c-c_m))\subseteq Z_{m'}\setminus Z_m$, hence $c_m$ defines a class in the relative homology module
$H_3(Z_{m'},Z_{m'}\setminus Z_m)$, while $\partial c_m$ defines a class in $H_2(Z_{m'}\setminus \inte (Z_{m}))$.   
We first decompose $\partial c_m$ into the contributions corresponding to the connected components of $Z_{m'}\setminus Z_m$.  
Let 
$B_1,\dots,B_k$ be the boundary components of $Z_m$, and denote
by $W_i$ the component of $Z_{m'}\setminus Z_m$ whose closure contains $B_i$. Since each component of  $M\setminus Z_m$ 
has exactly one boundary component,
we have $W_i\neq W_j$ for $i\neq j$.
We set $w_i=(\partial c_m)|_{W_i}$, so that
$\partial c_m=w_1+\dots +w_k$. Since the $W_i$ are pairwise disjoint, each $w_i$ is a cycle in $C_2(W_i)$. 


For every point $p\in \inte(Z_m)$, the chain $c_m$ defines the standard generator  $1\in H_3(Z_m,Z_m\setminus \{p\})\cong \mathbb{R}$. As a consequence,
under the excision isomorphism $H_3(Z_{m'},Z_{m'}\setminus Z_m)\cong H_3(Z_m,\partial Z_m)$, the class
$[c_m]$ corresponds to the relative fundamental class of $Z_m$. 
Since the connecting homomorphism $H_3(Z_m)\to H_2(\partial Z_m)$ takes $[Z_m]$ to the sum of the fundamental classes
of the components of $\partial Z_m$, it follows that $[w_i]=[B_i]$ in $H_2(W_i)$, where $[B_i]$ is the image of the fundamental
class of $B_i$ under the map $H_2(B_i)\to H_2(W_i)$ induced by the inclusion $B_i\hookrightarrow W_i$. 

Observe now that $\|w_i\|_1\leq \|\partial c_m\|_1<2$. By~\cite[Corollary 6.18]{Gabai}, this implies that the Thurston norm of $[B_i]=[w_i]$ is strictly smaller than $1$.
But $[B_i]$ is (the image in a real homology module of) an integral class, hence the Thurston norm of $[B_i]$ is in fact zero~\cite{Thurston:norm}, i.e.~$[B_i]$ may be represented by the union of a finite family
of spheres and tori. 
That is, there exists a finite family $\{S_1,\dots,S_h\}$ of disjoint spheres and tori in $W_i$ such that, if $[S_i]$ denotes the image of the fundamental class of $S_i$
in $H_2(W_i)$, then $[B_i]=[S_1]+\dots +[S_h]$. Let $W_i'$ be the  connected component of $W_i\setminus (S_1\cup\dots\cup S_h)$ containing
$B_i$; we claim that $\overline{W'_i}$ is  a compact manifold with boundary such that $B_i\subseteq \partial \overline{W'_i}\subseteq B_i\cup (S_1\cup\dots\cup S_h)$.

The fact that $\overline{W'_i}$ is  a compact manifold with boundary such that $B_i\subseteq \partial \overline{W'_i}$ is obvious, and
of course $ \partial \overline{W'_i}$ is contained in $\partial W_i\cup  (S_1\cup\dots\cup S_h)$, hence we only need to prove that the only component
of $\partial W_i$ contained in $\partial \overline{W'_i}$ is $B_i$. 
Let us denote by $i\colon H_2(W_i)\times H_1(W_i,\partial W_i)\to\mathbb{R}$ the usual intersection form. 
If $\gamma$ is a continuous path in $W_i$ joining $B_i$ with a component $C\neq B_i$
of $\partial W_i$, then it is easy to see that $I([B_i],[\gamma])=1$ (up to a slight perturbation, we may let $\gamma$ intersect $B_i$ transversely
only at its starting point), hence $I([S_1]+\dots+[S_h],[\gamma])=1$, which implies that $\gamma$ cannot be disjoint from $S_1\cup\dots\cup S_h$.
This shows that $\overline{W'_i}$ does not contain any component of
$\partial W_i\setminus B_i$, hence it is 
a compact manifold with boundary such that $B_i\subseteq \partial \overline{W'_i}\subseteq B_i\cup (S_1\cup\dots\cup S_h)$.

If we  set
$$
M'_n=Z_m\cup \left(\bigcup_{i=1}^k \overline{W'_i}\right)\ ,
$$
we then obtain the desired $3$-manifold bounded by spheres and tori, and we finally 
define $M_n$ as the union of $M'_n$ with the bounded components
of $M\setminus M_n'$. Up to choosing a subsequence of $\{M_n\}_{n\in\mathbb{N}}$, we may  also suppose that
$M_n\subseteq \inte(M_{n+1})$ for every $n\in\mathbb{N}$, thus obtaining the desired exhaustion of $M$.

\end{proof}

\begin{rem}
In any topological space,
integral homology classes of dimension $2$ may be represented by surfaces. 
In the proof of Proposition~\ref{toric} we used the fact that, in the context of $3$-manifolds,   a $2$-dimensional class with vanishing $\ell^1$-seminorm may be represented by a union of spheres and tori. However, this fact is very far from being true in general:
for example, it is shown in~\cite{BarGhys} that, for every $g\in\mathbb{N}$, there exist a space $X_g$ with with a nilpotent fundamental group and a class
$\alpha_g\in H_2(X_g,\mathbb{R})$  which does not lie in the subspace generated by classes of
$H_2(X_g,\mathbb{R})$ represented by surfaces of 
 genus smaller than $g$. However, since $\pi_1(X_g)$ is amenable, the $\ell^1$-norm of any class in $H_2(X_g,\mathbb{R})$ vanishes. 
 
 In general, characterizing $2$-dimensional
 classes with vanishing seminorm seems to be a challenging task, we refer the reader to~\cite{FS2020} for a discussion of this topic.
\end{rem}

\begin{lemma}\label{solidtori}
Let $M$ be an open irreducible $3$-manifold, and let $N\subseteq M$ be a compact connected codimension-0 submanifold with boundary. Also suppose that
every component of $M\setminus N$ is unbounded. Then:
\begin{enumerate}
\item If one component of $\partial N$ is a sphere, then $N$ is homeomorphic to a closed $3$-ball.
\item If one component of $\partial N$ is a compressible torus, then $N$ is homeomorphic to a solid torus.
\end{enumerate}
\end{lemma}
\begin{proof}
Let $S\subseteq \partial N$ be a component of $\partial N$. By definition of irreducible manifold, if $S\cong S^2$, then $S$ bounds a closed $3$-ball $B\subseteq M$. Moreover,
it is well-known that every compressible torus in an irreducible $3$-manifold bounds a solid torus, hence 
if $S\cong S^1\times S^1$ and $S$ is compressible, then $S=\partial T$, where $T\subseteq M$ is a solid torus. Let us set $W=B$ in the former case,
and $W=T$ in the latter one. 

We first show that $W\subseteq N$. Since $W$ is compact,
  $W\cap (M\setminus N)$ is closed in $M\setminus N$. Moreover, since $\partial W\subseteq N$, we have
  $W\cap (M\setminus N)=\inte(W)\cap (M\setminus N)$. Since $\inte(W)$ is a $3$-manifold without boundary, by the Domain Invariance Theorem we have that
   $W\cap (M\setminus N)$ is also open in $M\setminus N$. Therefore, $W\cap (M\setminus N)$ is a union of connected components of $M\setminus N$. However, since 
   no component of $M\setminus N$ is bounded, while $W$ is compact, this implies that $W\cap (M\setminus N)=\emptyset$, i.e.~that $W\subseteq N$.
   
 Now the inclusion $W\hookrightarrow N$ extends to an inclusion between the doubles $DW\hookrightarrow DN$ of the manifolds $W,N$.
 Just as above, $DW$ is open (by the Domain Invariance Theorem) and closed (since it is compact) in $DN$. Since $N$ is connected,
 this implies $DW=DN$, whence $W=N$, as desired. 
   \end{proof}


\begin{prop}\label{irreducible:prop}
Let $M$ be an irreducible $3$-manifold not homeomorphic to $\mathbb{R}^3$ and such that $\|M\|<+\infty$. Then $M$ admits an exhaustion $\{M_n\}_{n\in\mathbb{N}}$ such that
$M_0$ is not contained in any $3$-ball in $M$, and at least one of the following conditions holds:
\begin{enumerate}
\item Each $M_n$, $n\in\mathbb{N}$, is homeomorphic to a solid torus and, for every $n\in\mathbb{N}$, the torus $\partial M_n$ is incompressible in $M\setminus \inte (M_0)$, or
\item For every $n\in\mathbb{N}$, the boundary $\partial M_n$ consists of a finite union of incompressible tori.
\end{enumerate} 
\end{prop}
\begin{proof}
By Proposition~\ref{toric}, 
there exists an exhaustion  $\{M_n\}_{n\in\mathbb{N}}$ of $M$ such that
$\partial M_n$ is a finite union of spheres and/or tori  for every $n\in\mathbb{N}$.

Suppose that every $M_n$ is contained in a closed $3$-ball $B_n\subseteq M$. Then $M$ is an increasing union of $3$-balls,  
hence it is homeomorphic to $\mathbb{R}^3$~\cite{Browncells}, against our hypothesis. Up to discarding a finite number of $M_n$, we may thus assume that
$M_0$ is not contained in any $3$-ball in $M$. As a consequence, no $M_n$ is a closed $3$-ball, hence by Lemma~\ref{solidtori} all the boundary components
of every $M_n$ are tori.

Up to passing to a subsequence of $\{M_n\}_{n\in\mathbb{N}}$, if necessary, we may assume that either all the components of $\partial M_n$ are incompressible for every $n\in\mathbb{N}$, or every $M_n$ has at least one boundary component which is a compressible torus. In the former case, condition (2) of the statement is fulfilled, and we are done.

In the latter case, Lemma~\ref{solidtori} implies  that every $M_n$ is a solid torus.
In order to conclude, suppose by contradiction that there exists $n\in\mathbb{N}$ such that  $\partial M_n$ is compressible in $M\setminus \inte(M_0)$, and let
$D\subseteq M\setminus M_0$ be a compression disc for $\partial M_n$. Let also $N$ be obtained from $M$ by adding (if $D\subseteq M\setminus \inte(M_n)$)
of removing (if $D\subseteq M_n$) a small tubular neighbourhood of $D$ avoiding $M_0$. By construction $\partial N$ is a sphere, hence by Lemma~\ref{solidtori}
$N$ is a closed $3$-ball. But this contradicts the fact that $M_0$ is not contained in any $3$-ball, thus concluding the proof.
\end{proof}

Classically, an open $3$-manifold $M$ is called \emph{eventually end irreducible} if there exist a compact subset $K\subseteq M$ and an exhaustion 
$\{M_n\}_{n\in\mathbb{N}}$ of $M$
such that every component of $\partial M_n$ is incompressible in $M\setminus K$ (see e.g.~\cite{embrown}). According to this terminology, Proposition~\ref{irreducible:prop} ensures
that an open irreducible $3$-manifold with finite simplicial volume is eventually end irreducible. Eventually end irreducible manifolds are somewhat more manageable than generic open $3$-manifolds. However, as observed in~\cite[Figure 1]{embrown}, it is possible to construct a contractible open $3$-manifold (which, moreover, embeds in $\mathbb{R}^3$) 
which is not eventually end irreducible.

In order to apply the results proved so far to contractible manifolds, we need the following lemma.

\begin{lemma}\label{irreducible}
Let $M$ be an open contractible $3$-manifold. Then $M$ is irreducible.
\end{lemma}

\begin{proof}
Let $S$ be a $2$-sphere in $M$. Since $S$ is nullhomologous, $M\setminus S$ is the union of two connected components
$M_1$, $M_2$.
Since a contractible manifold has only one topological end~\cite[Proposition 3.2]{Stallings:end}, 
at least one of the $M_i$, say $M_1$, must be relatively compact. 
By Van Kampen Theorem, we know that $\pi_1(\overline{M}_i)\cong \pi_1(M_i)=\{1\}$ for $i=1,2$. 
Therefore, $\overline{M}_1$ is a compact simply connected manifold whose boundary consists of a $2$-sphere. The Poincar\'e Conjecture now implies that there do not
exist fake $3$-cells, hence $\overline{M}_1$ is homeomorphic to a closed $3$-ball. This shows that $M$ is irreducible.
\end{proof}

\begin{cor}\label{contractible:cor}
Let $M$ be a contractible $3$-manifold not homeomorphic to $\mathbb{R}^3$ and such that $\|M\|<+\infty$. Then $M$ admits an exhaustion $\{M_n\}_{n\in\mathbb{N}}$ such that the following conditions hold:
\begin{enumerate}
\item $M_0$ is not contained in any $3$-ball in $M$;
\item Each $M_n$, $n\in\mathbb{N}$, is homeomorphic to a solid torus;
\item For every $n\in\mathbb{N}$, the torus $\partial M_n$ is incompressible in $M\setminus \inte (M_0)$;
\item For every $n\in\mathbb{N}$, the inclusion $M_n\hookrightarrow M_{n+1}$ induces the trivial map $\pi_1(M_n)\to\pi_1(M_{n+1})$.
\end{enumerate} 
\end{cor}
\begin{proof}
By Lemma~\ref{irreducible}, the manifold $M$ is irreducible. Moreover, being simply connected, it cannot contain incompressible tori, hence 
by Proposition~\ref{irreducible:prop} 
there exists an exhaustion  $\{M_n\}_{n\in\mathbb{N}}$ of $M$ satisfying conditions (1), (2) and (3) of the statement. 

In order to prove (4), let $\gamma_n$ be a generator of $\pi_1(M_n)\cong\mathbb{Z}$. Since $M$ is simply connected and the image of any homotopy between a representative of $\gamma$ and the constant map is compact, there exists $i(n)\geq n$ such that the inclusion $M_n\hookrightarrow M_{i(n)}$ induces the trivial map $\pi_1(M_n)\to\pi_1(M_{i(n)})$. Therefore, up to replacing $\{M_n\}_{n\in\mathbb{N}}$ with an exhaustion $\{M_{j(n)}\}_{n\in\mathbb{N}}$, where $j\colon\mathbb{N}\to \mathbb{N}$ is strictly increasing,
we may assume that also condition (4) of the statement holds. 
\end{proof}

\begin{prop}\label{exhaustion}
Let $\{M_n\}_{n\in\mathbb{N}}$ be an exhaustion of the open $3$-manifold $M$, and for every 
$n\in\mathbb{N}$ let $W_n=M_n\setminus \inte{M_{n-1}}$, where we understand that $M_{-1}=\emptyset$ (hence $W_0=M_0$).
Suppose that the following conditions hold:
\begin{enumerate}
\item for every $n\in\mathbb{N}$,
$\partial M_n$ is a finite union of spheres and/or tori;
\item for every $n\in\mathbb{N}$, every toric component of $\partial M_n$ is incompressible in $M\setminus \inte (M_0)$;
\item 
$\|W_n,\partial W_n\|> 0$ for infinitely many $n\in\mathbb{N}$.
\end{enumerate}
Then  $\|M\|=+\infty$.
\end{prop}
\begin{proof}
We need to show that $\|z\|_1=+\infty$ for every locally finite fundamental cycle for $M$. 
Thus, let $z$ be such a cycle and
for every $n\in\mathbb{N}$ set 
$z_n=z|_{W_{n}}$.

We define a strictly increasing sequence $\{n_i\}_{i\in\mathbb{N}}\subseteq \mathbb{N}$
as follows. First of all, observe that $\supp(z_0)$ is compact, hence there exists ${n}$ such that $\supp(z_0)\subseteq \inte(M_{{n}})$.
Choose $n_0>{n}$ such that $\|W_{n_0},\partial W_{n_0}\|>0$. Since $W_{n_0}$ is compact, $\supp(z_{n_0})$ is also compact,
hence there exists (a new) $n>n_0$ such that $\supp(z_{n_0})\subseteq  \inte(M_n)$ (hence, $\supp(z_0)\cup \supp(z_{n_0})\subseteq  \inte(M_n)$). We then choose $n_1>n$ such that $\|W_{n_1},\partial W_{n_1}\|>0$. We may now iterate this
 contruction as follows: once $n_0,n_1,\dots,n_i$ have been chosen, there exists $n>n_i$ such that $\supp(z_0)\cup \supp(z_{n_0})\cup \dots \cup \supp(z_{n_i})\subseteq \inte(M_n)$.
 We then choose $n_{i+1}>n$ such that $\|W_{n_{i+1}},\partial W_{n_{i+1}}\|>0$.
 
By construction, no singular simplex may appear with a non-null coefficient both in $z_{n_i}$ and in $z_{n_j}$ for $i\neq j$,
hence
\begin{equation}\label{somma:eq}
\|z\|_1\geq \sum_{i=0}^\infty \|z_{n_i}\|_1\ .
\end{equation}

In order to conclude it is thus sufficient to show that $\|z_{n_i}\|_1\geq k$ for every $i\in\mathbb{N}$,
where $k$ is the constant provided by Lemma~\ref{general2:lemma}.
 Let us fix $i\in\mathbb{N}$, and let $N_i$ be  a component
of $W_i$ with $\|N_i,\partial N_i\|>0$ (hence, $\|N_i,\partial N_i\|\geq k$). 
We set $M'=M\setminus \inte(M_0)$. 
By construction, the chain $z_{n_i}$ is supported in $M'$. Moreover,
since $\partial z=0$ we have $\partial z_{n_i}=-\partial (z-z_{n_i})$, hence $\partial z_{n_i}$ is supported outside $N_i$. As a consequence, if 
$z'_{n_i}$ is the image of $z_{n_i}$ under the quotient map $C_3(M')\to C_3(M', M'\setminus \inte(N_i))$, then
$z'_{n_i}$ defines a class $[z'_{n_i}]\in H_3(M', M'\setminus\inte(N_i))$. 

Let now $\psi_3\colon H_3(N_i,\partial N_i)\to H_3(M', M'\setminus \inte(N_i))$ be the map induced by the inclusion. It is immediate
to realize that, for every $x\in \inte(N_i)$, the chain $z_{n_i}$ projects to the positive generator of $H_3(M,M\setminus\{x\})$, and this suffices to show that
$\psi_3([N_i,\partial N_i])=[z'_{n_i}]$. Thanks to the assumptions in the statement, Proposition~\ref{amenable:prop} (applied to the submanifold
$N_i$ of the ambient manifold $M'$) now implies that the map $\psi_3$ is an isometric isomorphism, hence 
$$
\|z_{n_i}\|_1\geq \|z'_{n_i}\|_1\geq \|[z'_{n_i}]\|_1=\|\psi_3([N_i,\partial N_i])\|_1=\|[N_i,\partial N_i]\|_1\geq k\ .
$$
This concludes the proof.
\end{proof}

\section{On $2$-component graph links}\label{link:sec}
In this paper we will consider only knots and links in $S^3$, and every knot and link will be oriented. If
$L= K_1 \cup \ldots \cup K_n \subset S^3$ is an $n$-component link, then we denote by $N(L)=N(K_1) \cup \ldots \cup N(K_n)$ a closed tubular neighborhood of $L$,
and by $X(L)=S^3 \setminus \inte(N(L))$ the \textit{link exterior} of $L$. 
If $K_1,K_2$ are disjoint knots in $S^3$, we denote by $\lk(K_1,K_2)$ the \emph{linking number} of $K_1$ and $K_2$ (we refer the reader e.g.~to~\cite{rolfsen} for some standard terminology in knot theory).
The main result of this section is the following:

\begin{thm}\label{lkgraphlink}
Let $L=K_1 \cup K_2$ be a non-split 2-component link in $S^3$ such that:
\begin{itemize}
    \item[(1)] $K_1$ is the unknot;
    \item[(2)] $\|X(L),\partial X(L)\|=0$.
\end{itemize}
Then $\lk(K_1, K_2) \neq 0$.
\end{thm}

We begin by recalling some results about graph links. Following~\cite{EN85}, we say that a link $L$ is \emph{graph link} if and only if the link exterior $X(L)$ is a connected
sum of graph manifolds (see e.g.~\cite{AFW} for some standard terminology in $3$-manifold theory). Since the simplicial volume is additive with respect
to connected sums and gluings along incompressible tori~\cite{Gromov, BBFIPP}, this condition is equivalent to the fact that
$\|X(L),\partial X(L)\|=0$ (see e.g.~\cite{Soma}).
We now define two operations on links:
    
    \smallskip
    
    {\bf Connected sum}: Let $L=K_1\cup \ldots \cup K_n$ and $L'=K_1' \cup \ldots K_m'$ be disjoint oriented links in $S^3$ and let $\Sigma \subset S^3$ be a $2$-sphere; let $C_1$ and $C_2$ be the two connected components in which $\Sigma$ separates $S^3$. We can isotope $L$ and $L'$ so that they are contained respectively in $C_1$ and $C_2$. Let us fix one component on each link, say $K_1$ and $K_1'$. Let $Q \simeq I \times I$ be a square such that $Q \cap \Sigma \simeq \{\frac{1}{2}\} \times I$, $Q \cap L \simeq \{0\} \times I \subset K_1$ and $Q \cap L' \simeq \{1\} \times I \subset K_1'$, in such a way that the boundary orientation of $\partial Q$ agrees with the ones of $K_1$ and $K_1'$ along
    $Q\cap K_1$ and $Q\cap K'_1$. Then, the closure of $(K_1 \cup \partial Q \cup K_1')\setminus (Q\cap (K_1\cup K'_1))$ is isotopic to the connected sum $K_1 \# K_1'$. The connected sum of the links $L$ and $L'$ along the components $K_1$ and $K_1'$ is the resulting link $L\# L'= K_1 \# K_1' \cup K_2 \ldots \cup K_n \cup K_2' \ldots \cup K_m'$.
    \smallskip
    
   {\bf Cabling}:  Let $L=K_1\cup \ldots \cup K_n$ be a link in $S^3$. Let $i \in \{1, \ldots, n\}$  and let $N_i$ denote a tubular neighborhood of $K_i$. Let also $\mu_i,\lambda_i\subseteq \partial N_i$ be the  meridian and the longitude of the link component $K_i$. Let $(p,q)$ be a pair of coprime integers different from $(0,0)$ (we understand that 0 and 1 are coprime). Let $d$ be a natural number. We denote by $dK_i(p,q)$ the union of $d$ parallel curves on $\partial N_i$, each isotopic in $\partial N_i$ to $p\lambda_i+\ q\mu_i$. The operation of replacing $L$ with either $L \cup dK_i(p,q)$ or with $L \cup dK_i(p,q) \setminus K_i$ is called a $(p,q)$-cabling on $L$ along $K_i$ (respectively without remotion of the core or with remotion of the core). 

 Let us now define recursively certain classes of links in $S^3$:

\begin{itemize}
    \item $\mathcal{C}_0$ is the class that contains only the unknot;
    \item a link $L$ belongs to $\mathcal{C}_j$ if and only if it is obtained via a cabling of an element in $\mathcal{C}_i$ with $i<j$ or is obtained as a connected sum of $L_1 \in \mathcal{C}_{j_1}$ and $L_2 \in \mathcal{C}_{j_2}$, with $j_1, j_2 < j$;
    \item  $\mathcal{C}= \bigcup_{i=0}^{+\infty}\mathcal{C}_{i}$.
\end{itemize}

The following characterization of graph links due to Eisenbud and Neumann plays a crucial r\^ole in the proof of Theorem~\ref{lkgraphlink}.

\begin{thm}[{\cite[Theorem 9.2]{EN85}}]\label{grlksolv}
$L$ is a graph link if and only if  $L \in \mathcal{C}$.
\end{thm}

\begin{proof}[Proof of Theorem~\ref{lkgraphlink}]
Let $L$ be any link satisfying the assumptions of the theorem. 
As observed above, the vanishing of $\|X(L),\partial X(L)\|$ implies that $L$ is a graph link, hence $L\in \mathcal{C}$ by Theorem~\ref{grlksolv}.
Therefore, in order to conclude it is sufficient to show, by induction on $i\in\mathbb{N}$, that the following holds: 
if $L=K_1\cup K_2$ is a non-split 2-component link in $\mathcal{C}_i$ such that $K_1$
 is the unknot, then $\lk(K_1,K_2)\neq 0$.

For $i=0$ there is nothing to prove, since $\mathcal{C}_0$ does not contain $2$-component links. Let us now suppose  $i \geq 1$.  
By definition of $\mathcal{C}_i$, the link
$L$ can be obtained in three different ways:

\begin{itemize}
    \item[(1)] via a $(p,q)$-cabling with $d=2$ and with remotion of the core along a knot $K \in \mathcal{C}_{j}$, $j<i$;
    \item[(2)] via a $(p,q)$-cabling with $d=1$ and without remotion of the core along a knot $K \in \mathcal{C}_{j}$, $j<i$;
    \item[(3)] via a connected sum of a link $L' \in \mathcal{C}_{j_1}$ and a knot $K' \in \mathcal{C}_{j_2}$ with both $j_1, j_2< i$.
\end{itemize}

In the first case,  we have $(p,q)\neq (0,1)$, because otherwise $L$ would be split. Therefore, $p\neq 0$, and each component of 
$L$ is a satellite of $K$. Since $K_1$ is the unknot and any satellite of a non-trivial knot is non-trivial, we then deduce that $K$ is the unknot. 
Thus $(p,q)\neq (1,0)$, because otherwise $L$ would be split, and $L$ is a torus link. 
As a consequence we have $\lk (K_1,K_2)\neq 0$, as desired.


In the second case we have $L=K\cup K(p,q)$. We first observe that $(p,q)\neq (1,0)$,
because otherwise either $L$ would be split (if $K$ is the unknot), or no component of $L$ would be trivial (if $K$ is non-trivial). Therefore, $q\neq 0$ and
$\lk(K_1,K_2)=\lk(K,K(p,q))=q\neq 0$, as desired.

In the third case, first observe that $L'$ is non-split (otherwise, also $L$ would be split), and that at least one component of $L'$ is the trivial knot  (since a connected sum of knots
is trivial only if both summands are trivial). Therefore, we can apply the inductive hypothesis to $L'=J_1 \cup J_2$ and deduce that
 $\lk(J_1, J_2) \neq 0$. But it is immediate to check that $\lk(K_1,K_2)=\lk(J_1,J_2)$, hence $\lk(K_1,K_2)\neq 0$, and this concludes the proof. 
 \end{proof}

\section{Proof of Theorem~\ref{main:thm}}\label{proof:sec}

We now prove Theorem~\ref{main:thm} arguing by contradiction. 
Let $M$ be a contractible  $3$-manifold not homeomorphic to $\mathbb{R}^3$ and 
suppose that $\|M\|<+\infty$.

Let $\{M_n\}_{n \in \mathbb{N}}$ be an exhaustion of $M$ satisfying the properties described in Corollary~\ref{contractible:cor}. 
For every 
$n\in\mathbb{N}$ let $W_n=M_n\setminus \inte(M_{n-1})$, where we understand that $M_{-1}=\emptyset$ (hence $W_0=M_0$). 
By Proposition~\ref{exhaustion}, 
 there exists $n_0 \in \mathbb{N}$ such that $\|W_n, \partial W_n\|=0$ for every $n>n_0$. Let us now fix $n>n_0$. Recall that $M_n$ and $M_{n-1}$ are solid tori.
 As a consequence, there exists a homeomorphism $\varphi\colon M_n\to X(K_2)$ between $M_n$ and the exterior of the unknot $K_2\subseteq S^3$. If $K_1\subseteq S^3$ is the image of the core of $M_{n-1}$ via $\varphi$, then $\varphi$ restricts to a homeomorphism between $W_n$ and $X(L)$, where $L=K_1\cup K_2\subseteq S^3$. 
 Since $M_0$  is not contained in any $3$-ball in $M$, the knot $K_1$ is not contained in any $3$-ball in $W_n$, and this implies that $L$ is not split. 
 Moreover, $\|X(L),\partial X(L)\|=\|W_n,\partial W_n\|=0$, thus Theorem~\ref{lkgraphlink} implies that $\lk(K_1,K_2)\neq 0$. This contradicts the fact that
the inclusion $M_{n-1}\hookrightarrow M_{n}$ induces the trivial map $\pi_1(M_{n-1})\to\pi_1(M_{n})$, and concludes the proof of the theorem.

\section{The higher dimensional case}\label{higher:sec}
In this section we prove Theorem~\ref{higher:dim:thm}, i.e.~we show that, if $n\geq 4$, then
there exists a smooth contractible $n$-manifold which is not homeomorphic to $\mathbb{R}^n$ and is such that $\|M\|=0$.

Let us first suppose $n=4$. It is shown in~\cite{Casson} that there exist Seifert fibered $3$-manifolds distinct from $S^3$ which bound
 Mazur $4$-manifolds. In other words, there exists a contractible, compact, smooth four-dimensional manifold $W$ 
such that $\partial W$ is a connected Seifert fibered $3$-manifold not homeomorphic to $S^3$.

Let $M=\inte(W)$. Using that $\partial W$ admits a collar in $W$ and that $\pi_1(\partial W)\neq \{1\}$,
it is immediate to check that $M$ is not simply connected at infinity, i.e.~that for every compact subset $K\subseteq M$ there exists a compact
subset $K'\subseteq M$ such that $K\subseteq K'$ and $M\setminus K'$ is not simply conneted. This implies in turn that $M$ cannot be homeomorphic to $\mathbb{R}^4$.
On the other hand, it is proved in~\cite{Loeh} that Seifert fibered manifolds are \emph{$\ell^1$-invisible}, and that the simplicial volume of the internal part of a compact manifold with $\ell^1$-invisible boundary is finite. As a consequence, we have $\|M\|<+\infty$. However, since $M$ is simply connected we have $\|M\|\in \{0,+\infty\}$ (see again~\cite{Loeh}),
hence $\|M\|=0$.

The proof in dimension $n\geq 5$ is similar but easier, since in  higher dimension it is (relatively) easier to construct smooth contractible manifolds which are bounded by non-simply connected
$\ell^1$-invisible manifolds. Let $\Gamma$ be the fundamental group of the Poincar\'e dodecahedral homology $3$-sphere. The group
$\Gamma$ has order 120, is perfect, and has deficiency 0. By~\cite[Theorem 1 and Remark at page 70]{Kerv}, for every $n\geq 4$ there exists a smooth
homology $n$-sphere $Z_n$ such that $\pi_1(Z_n)=\Gamma$. By~\cite[Theorem 3]{Kerv}, this implies in turn that for every $n\geq 5$
there exists a smooth contractible compact $n$-manifold $W_n$ such that $\partial W_n$ is connected and $\pi_1(\partial W_n)=\Gamma$. Since $\Gamma$ is finite, it is amenable,
hence $\partial W_n$ is $\ell^1$-invisible~\cite{Loeh}. Just as above, this implies  that $\|\inte(W_n)\|=0$.
Moreover, $\inte(W_n)$ is not simply connected at infinity, hence it is not homeomorphic to $\mathbb{R}^n$. This concludes the proof of Theorem~\ref{higher:dim:thm}.

\section{The spectrum of the simplicial volume for irreducible $3$-manifolds}\label{spectrum:sec}

In this section we prove Theorem~\ref{spectrum:thm}, which states that $$SV^{\text{lf}}_{\text{irr}}(3)\subseteq SV(3)\cup \{\infty\}\ .$$

By~\cite{KimKue}, if $N$ is a compact $3$-manifold  whose boundary is given by spheres and/or tori, then 
$\|\inte(N)\|=\|N,\partial N\|$, hence $\|N,\partial N\|\in SV^{\text{lf}}_{\text{tame}}(3)\subseteq SV(3)\cup \{\infty\}$. Therefore, it is sufficient to show that, if
$M$ is an open irreducible $3$-manifold with $\|M\|<+\infty$, then $\|M\|=\|N,\partial N\|$ for some compact $3$-manifold  $N$ whose boundary is given by spheres and/or tori.

Suppose that $M$ is an open irreducible $3$-manifold with $\|M\|<+\infty$. We may also suppose that $M$ is not homeomorphic to $\mathbb{R}^3$, otherwise
$\|M\|=0$ and we are done.
Let  $\{M_n\}_{n\in\mathbb{N}}$ be the exhaustion provided by Proposition~\ref{irreducible:prop}, and 
 for every 
$n\in\mathbb{N}$ let $W_n=M_n\setminus \inte{M_{n-1}}$, where as usual $M_{-1}=\emptyset$.
By Proposition~\ref{exhaustion}, there exists $n_0\in\mathbb{N}$ such that $\|W_n,\partial W_n\|=0$ for every $n>n_0$.  
Up to replacing  $M_{n}$ (resp.~$W_n$) with $M_{n_0+n}$ (resp~$W_{n_0+n}$) 
for every $n\in\mathbb{N}$, we may then assume that $\|W_n,\partial W_n\|=0$ for every $n>0$.
In order to conclude the proof, we will show that 
 $$ \|M\|= \|M_{0},\partial M_{0}\|\ .$$

We first show that $ \|M\|\leq \|M_{0},\partial M_{0}\|$.
Recall that, for every $n\in\mathbb{N}$, every component of $\partial M_n$ (hence, of $\partial W_n$) is homeomorphic to a torus.
Since the torus has an amenable fundamental group,
its second homology group satisfies
the uniform boundary condition~\cite{Matsu-Mor}, i.e.~there exists a constant $\theta>0$ such that, if 
$b\in C_2(S^1\times S^1)$ 
is a boundary, then there exists a chain
$c\in C_3(S^1\times S^1)$ such that $b=\partial c$ and $\|c\|_1\leq \theta\cdot \|b\|_1$. 

Let now $\varepsilon>0$ be given. 
Since the boundary components of $M_{0}$ are  tori, Gromov's Equivalence Theorem~\cite{Gromov, BBFIPP} implies that $M_{0}$ admits a fundamental
cycle $z_{0}\in C_3(M_{0},\partial M_{0})$ such that $\|z_{0}\|_1\leq \|M_{0},\partial M_{0}\|+\varepsilon/4$ and $\|\partial z_{0}\|\leq \varepsilon/(4\theta)$.
Moreover, Lemma~\ref{general1:lemma} implies that, for every $n>0$, the compact manifold $W_n$ admits a fundamental cycle $z_n\in C_3(W_n,\partial W_n)$ such that 
$\|z_n\|\leq \varepsilon/2^{n+2}$ and $\|\partial z_n\|\leq \varepsilon / (2^{n+3}\theta)$.  

Observe now that, for every $n\in\mathbb{N}$, we have $$\partial W_n=(\partial W_n\cap \partial W_{n-1})\cup (\partial W_n\cap \partial W_{n+1})=\partial M_{n-1}\cup \partial M_n\ .$$
(Recall that $M_{-1}=W_{-1}=\emptyset$, and $M_0=W_0$). 
For every $n\in\mathbb{N}$ we write $\partial z_n=c_n+d_n$, where $c_n$ is a sum of fundamental cycles of the components of $\partial M_{n-1}$, and
$d_n$ is a sum of fundamental cycles of the components of $\partial M_{n}$ (hence, $\partial z_0=d_0$ and $c_0=0$). Of course we have $\|\partial z_n\|_1=\|c_n\|_1+\|d_n\|_1$.
Since $W_n$ and $W_{n+1}$ induce
opposite orientations on $\partial M_n$, we have that $d_n+c_{n+1}$ is a boundary in $C_2(\partial M_n)$. Therefore, there exists $h_n\in C_3(\partial M_n)$ such that $\partial h_n= c_{n+1}+d_n$
and $$\|h_n\|_1\leq \theta(\|d_n+c_{n+1}\|_1)\leq \theta(\|d_n\|_1+\|c_{n+1}\|_1)\ .$$
Let now
$$
z=\sum_{n=0}^\infty z_n - \sum_{n=0}^\infty h_n\ .
$$
Being a locally finite sum of finite chains, $z$ is a locally finite chain. Moreover, 
$$
\partial z=\sum_{n=0}^\infty \partial z_n - \sum_{n=0}^\infty \partial h_n=\sum_{n=0}^\infty (c_n+d_n) - \sum_{n=0}^\infty (c_{n+1}+d_n)=c_0=0\ ,
$$
hence $z$ is a cycle. It is immediate to realize that $z$ is in fact a fundamental cycle for $M$. 
Finally, we have 
\begin{align*}
\|z\|_1&\leq \sum_{n=0}^\infty \|z_n\|_1 + \sum_{n=0}^\infty \|h_n\|_1\leq \|z_0\|+\sum_{n=0}^\infty \frac{\varepsilon}{2^{n+2}} + \sum_{n=0}^\infty \theta(\|c_{n+1}+d_n\|_1)\\
& \leq \left(\|M_{0},\partial M_{0}\|+\frac{\varepsilon}{4}\right)+\frac{\varepsilon}{2}+\theta\sum_{n=0}^{\infty} (\|c_n\|_1 + \|d_n\|_1)\\
& \leq \|M_{0},\partial M_{0}\|+\frac{3}{4}\varepsilon +\theta\sum_{n=0}^\infty \|\partial z_n\|_1\\
& \leq \|M_{0},\partial M_{0}\|+\frac{3}{4}\varepsilon +\theta\sum_{n=0}^\infty \frac{\varepsilon}{\theta 2^{n+3}}\\
& \leq \|M_{0},\partial M_{0}\| +\varepsilon\ .
\end{align*}
Due to the arbitrariness of $\varepsilon$, this implies that $\|M\|\leq  \|M_{0},\partial M_{0}\|$. 

In order to conclude, we now need to show that $\|M\|\geq  \|M_{0},\partial M_{0}\|$. If $M_0$ is a solid torus, then
$\|M_0,\partial M_0\|=0$, and there is nothing to prove. By Proposition~\ref{irreducible:prop}, we may then assume
that $\partial M_0$ is a finite union of incompressible tori. We can then argue as in the proof of Proposition~\ref{exhaustion}.
Let $z$ be a locally finite fundamental cycle for $M$, and let $z_0=z|_{M_0}$. Then $\partial z_0$ is supported in $M\setminus M_0$, and if
$z'_0$ 
is the image of $z_{0}$ under the quotient map $C_3(M)\to C_3(M, M\setminus\inte(M_0))$, then
$z'_{0}$ defines a class $[z'_{0}]\in H_3(M, M\setminus\inte(M_0))$. 

Let now $\psi\colon H_3(M_0,\partial M_0)\to H_3(M, M\setminus\inte(M_0))$ be the map induced by the inclusion. Then
$\psi_3([M_0,\partial M_0])=[z'_{0}]$. Since $\partial M_0$ is incompressible, Proposition~\ref{amenable:prop} (applied to the submanifold
$M_0$ of the ambient manifold $M$) now implies that the map $\psi_3$ is an isometric isomorphism, hence 
$$
\|z_{0}\|_1\geq \|z'_{0}\|_1\geq \|\psi([M_0,\partial M_0])\|_1=\|[M_0,\partial M_0]\|_1=\|M_0,\partial M_0\|\ .
$$
This concludes the proof of Theorem~\ref{spectrum:thm}.

\bibliography{biblio_contractible}
\bibliographystyle{alpha}

\end{document}